\documentclass[11pt]{article}
\usepackage{amsmath}
\usepackage{latexsym}
\usepackage{epsfig}
\usepackage{amssymb}
\begin{document}
\makeatletter
\def\@begintheorem#1#2{\trivlist \item[\hskip \labelsep{\bf #2\ #1.}] \it}
\def\@opargbegintheorem#1#2#3{\trivlist \item[\hskip \labelsep{\bf #2\ #1}\ {\rm (#3).}]\it}
\makeatother
\newtheorem{thm}{Theorem}[section]
\newtheorem{alg}[thm]{Algorithm}
\newtheorem{conj}[thm]{Conjecture}
\newtheorem{lemma}[thm]{Lemma}
\newtheorem{defn}[thm]{Definition}
\newtheorem{cor}[thm]{Corollary}
\newtheorem{exam}[thm]{Example}
\newtheorem{prop}[thm]{Proposition}
\newenvironment{proof}{{\sc Proof}.}{\rule{3mm}{3mm}}

\title{On Factor-Invariant Graphs With Two Cycles}
\author{Brian Alspach\\School of Mathematical and Physical Sciences\\
University of Newcastle\\Callaghan, NSW 2308\\
Australia\\brian.alspach@newcastle.edu.au\\ \\
Ted Dobson\\Department of Mathematics\\University of Primorska\\
Koper, Slovenia\\ted.dobson@upr.si\\ \\
Afsaneh Khodadadpour\\Department of Mathematical Sciences\\
Isfahan University of Technology\\Isfahan, Iran\\a.khodadadpour@math.iut.ac.ir\\ \\
Primo\v{z} \v{S}parl\\Faculty of Education\\University of Ljubljana\\
Ljubljana, Slovenia\\Institute Andrej Maru\v{s}i\v{c}\\University of Primorska\\Koper, Slovenia\\
Institute of Mathematics, Physics and Mechanics\\Ljubljana, Slovenia\\
Primoz.Sparl@pef.uni-lj.si}

\maketitle

\begin{abstract} We classify trivalent vertex-transitive graphs whose edge sets have a partition
into a 2-factor composed of two cycles and a 1-factor that is invariant under the action of the
automorphism group.
\end{abstract}

{\it keywords}: invariant partition, vertex-transitive, generalized Petersen graph.

\section{Introduction}

All graphs in this paper have neither multiple edges nor loops.  A $d$-{\it factor} in a graph $X$
is a spanning subgraph in which each vertex has valency $d$.  A
$(d_1,d_2,\ldots,d_t)$-{\it factorization} of a graph $X$ is a partition of
the edge set $E(X)$ into $d_i$-factors for $i=1,2,\ldots,t$.

\begin{defn} {\em A vertex-transitive graph $X$ is called $\mathcal{F}(d_1,d_2,\ldots,d_t)$-}invar\\iant
{\em when it admits a $(d_1,d_2,\ldots,d_t)$-factorization $\mathcal{F}$ that is invariant under
$\mathrm{Aut}(X)$, that is, each factor of $\mathcal{F}$ is mapped to itself by every element of
$\mathrm{Aut}(X)$.}
\end{defn}

In December 2016 Bojan Mohar asked the first author if one can characterize $\mathcal{F}(1,2)$-invariant
graphs.  The $\mathcal{F}(1,2)$-invariant graphs for which the 2-factor is a Hamilton cycle were characterized
in \cite{A1}.  In this paper we consider the case that the 2-factor is composed of two cycles.  Lest the reader
be concerned that this is the second in a seemingly infinite sequence of papers dealing with an increasing
number of cycles comprising the 2-factor, we point out that one cycle and two cycles are unique situations.  
The invariant 1-factor is a set of chords when the 2-factor is a Hamilton cycle and it is the only time an
edge of the 1-factor may be a chord.  Similarly, when there are two cycles, the 1-factor forms a perfect
matching between the two cycles.  When the 2-factor consists of three or more cycles, then the 1-factor
is split amongst the cycles and the problem becomes more complicated.  Of course, the preceding comments
refer to connected graphs because it suffices to solve the problem for connected graphs. 
This follows because the components of a disconnected vertex-transitive graph are
mutually isomorphic.

\begin{prop} If $X$ is an $\mathcal{F}(1,2)$-invariant graph and $F$ is the corresponding
2-factor, then the cycles of $F$ have the same length.
\end{prop}
\begin{proof} The result follows because $\mathrm{Aut}(X)$ acts transitively on $V(X)$ and preserves
the partition of $V(X)$ corresponding to $F$.    \end{proof}

\medskip

For the rest of the paper, $X$ will be an $\mathcal{F}(1,2)$-invariant graph whose corresponding 2-factor is
composed of two $n$-cycles $C_1$ and $C_2$ joined by the 1-factor of the invariant partition.
Moreover, the vertices of $C_1$ are labelled $u_0,u_1,\ldots,u_{n-1}$ so that $[u_i,u_{i+1}]$
is an edge of $C_1$, $0\leq i\leq n-1$, where the subscript calculations are carried out modulo $n$. 
The vertices of $C_2$ are labelled $v_0,v_1,\ldots,v_{n-1}$ so that $[u_i,v_i]$, $0\leq i\leq n-1$,
are the edges of the 1-factor.  For convenience, we use $U$ and $V$ to denote the sets of
vertices of $C_1$ and $C_2$, respectively.

We are not able to specify the edges of $V$, but because of the preceding labelling, we know a
lot about any automorphism of $X$ that fixes $U$ setwise.  In particular, if $f\in\mathrm{Aut}(X)$
maps $U\mbox{ to }U$ and $f(u_i)=u_j$, then $f(v_i)=v_j$ must hold. This follows easily
because each vertex $u_i\in U$ has the unique neighbor $v_i\in V$.  We use the {\it faithful}
language here.  That is, we say the restriction of $f$ to $U$ is the {\it faithful restriction} of $f$,
and we say that $f$ is the {\it faithful extension} of its restriction to $U$. 

It is important to point out that throughout the paper we use $D_n$ to denote the dihedral group
of degree $n$ and order $2n$, where $D_n=\langle\rho,\tau\rangle$, $\rho^n=\tau^2=1\mbox{ and }
\tau\rho\tau=\rho^{-1}$.  As permutations, we have \[\rho=(u_0\;u_1\;\cdots\;u_{n-1})
\mbox{ and }\tau=(u_0)(u_1\;u_{n-1})\cdots(u_{(n-2)/2}\;u_{(n+2)/2})(u_{n/2}),\]
when $n$ is even.

\section{Main Theorem}

We give the main theorem in this section, but first we need to introduce the cast of characters
involved in the main result.  There are three families of graphs that appear and we define them all
in spite of the fact two of them are well-known families.

The generalized Petersen graph $\mathrm{GP}(n,k)$ is defined as follows.  Its vertex set is
$\{u_0,u_1,\ldots,u_{n-1}\}\cup\{v_0,v_1,\ldots,v_{n-1}\}$. (It is not an accident that we 
are using the same vertex notation as used for $X$.)  The edges are the same as the description
for $X$ earlier with the addition of the edges $[v_i,v_{i+k}]$, $0\leq i\leq n-1$, where subscripts
again are calculated modulo $n$.

We require a subfamily of the honeycomb toroidal graphs so that we give a definition based
on our requirements rather than the standard definition given in \cite{A3}.
The {\it honeycomb toridal graph} $\mathrm{HTG}(2,n,\ell)$ has vertex set
$\{u_0,u_1,\ldots,u_{n-1}\}\cup\{v_0,v_1,\ldots,v_{n-1}\}$, where $n\geq 4$, $0\leq
\ell<n$ and both $n$ and $\ell$ are even.  The edges $[u_i,u_{i+1}]$ and $[v_i,v_{i+1}]$
are present so that we have two
$n$-cycles.  The edges between the two cycles are: $[u_i,v_i]$ for $i$ odd and $[u_{i+\ell},v_i]$ for
$i$ even.  The latter edges are called {\it jump edges}.

The third family of graphs contains graphs such that the edges in $V$ are described by two
odd integers.  The graph  $\mathcal{M}_O(n,a,b)$, where $n\geq 4$ is even and both
$a\mbox{ and }b$ are odd with $0<a<b<n$, has vertex set
$\{u_0,u_1,\ldots,u_{n-1}\}\cup\{v_0,v_1,\ldots,v_{n-1}\}$.  The edges are
the same as the description for $X$ earlier together with the following edges in $V$.  For each
even value of $i$, it has the edges $[v_i,v_{i+a}]\mbox{ and }[v_i,v_{i+b}]$. 
\smallskip

We now describe a collection $\mathcal{C}$ of graphs via the following list.

\begin{itemize}
\item All generalized Petersen graphs
$\mathrm{GP}(n,k)$ for which $k^2\equiv\pm 1(\mbox{mod }n)$ belong to
$\mathcal{C}$ with the exception of those for which $(n,k)$ is in the following list:  $$(4,\pm 1),
(5,\pm 2),(8,\pm 3),(10,\pm 3),(12,\pm 5)\mbox{ and }(24,\pm 5).$$
\item The honeycomb toroidal graphs $\mathrm{HTG}(2,n,\pm\ell), 0<\ell<n/2$, belong to
$\mathcal{C}$ with the exception of those satisfying any one of the following conditions:
\begin{enumerate}
\item $\mathrm{gcd}(n,\ell+2)=4$ and $4n|(\ell^2+4\ell-12)$;
\item $\mathrm{gcd}(n,\ell-2)=4$ and $4n|(\ell^2-4\ell-12)$; and
\item $\mathrm{gcd}(n,\ell+2)=4=\mathrm{gcd}(n,\ell-2)$ and $4n|(\ell^2+12)$.
\end{enumerate}
\item The graphs $\mathcal{M}_O(n,1,b)$, $3<b<n-2$, belong to $\mathcal{C}$ if and
only if $\mathrm{gcd}(n,b-1)=2$, $(b-1)^2/2\equiv 2(\mbox{mod }n)$ and the parameters
do not satisfy $n\equiv 0(\mbox{mod }8)$ and $b=(n+6)/2$.
\item The graphs $\mathcal{M}_O(n,a,b)$ satisfying Theorem \ref{vt} and $1<a<b-2<
n-a-2$ belong to $\mathcal{C}$ if and only if the parameter set satisfies none of the following:
\begin{enumerate}
\item $n\equiv 8(\mbox{mod }16)$, $a=4a_0+1<n/4-1$, $a_0$ odd, such that $8(a_0^2+a_0+1)=
0$ in $\mathbb{Z}_n$, and $b=n/2+a+2$;
\item $n\equiv 8(\mbox{mod }16)$, $b=4b_0+1$, $b_0$ odd, such that $8(b_0^2+b_0+1)=
0$ in $\mathbb{Z}_n$, $n/2<b<(3n-4)/4$ and $a=b-n/2+2$;
\item $n\equiv 48(\mbox{mod }96)$ and either $a=n/4+3$ and $b=n/2+1$, or $a=n/4-3$ and
$b=n/2-1$;
\item $8|n$, $b=4b_0+3<(3n+4)/4$, $b_0>0$ even, $4(b_0+1)^2\equiv 4
(\mbox{mod }n)$ and $a=b-n/2-2$; or $b=4b_0+1<(3n-4)/4$, $b_0$ odd, $4b_0^2\equiv 4(\mbox{mod }n)$ and
$a=b-n/2+2$; and
\item $8|n$, $a=4a_0+3<(n+4)/4$, $a_0$ even, $4(a_0+1)^2\equiv 4(\mbox{mod }n)$
and $b=n/2+a-2$; or $a=4a_0+1<(n-4)/4$, $a_0$ odd, $4a_0^2\equiv 4(\mbox{mod }n)$
and $b=n/2+a+2$. 
\end{enumerate} 
\end{itemize}

\begin{thm}\label{main} A connected trivalent vertex-transitive graph $X$ is $\mathcal{F}(1,2)$-
invariant, where the 2-factor is composed of two cycles, if and only if it belongs to $\mathcal{C}$.
\end{thm} 

{\bf Outline of the proof of Theorem \ref{main}}.  Because of the length of the proof of the
preceding theorem---that is, the rest of the paper---we now give an outline of the proof.  The
graphs that arise in the proof are a result of the setwise stabilizer of $U$ being a transitive
subgroup of the dihedral group.  The generalized Petersen graphs occur when $\rho$ is in
the transitive subgroup.  When $n$ is even, there is a transitive subgroup of $D_n$ not
containing $\rho$ and this transitive subgroup produces graphs of the form $\mathcal{M}_O(n,a,b)$.

Generalized Petersen graphs are covered in Section 3.  Graphs of the form $\mathcal{M}_O(n,a,b)$
fill the remainder of the paper.  The basic strategy for the latter graphs is to eliminate those which
are not $\mathcal{F}(1,2)$-invariant.  There are three situations for which the elimination is slightly
convoluted and we shall point them out when they occur.  In Section 4 we determine conditions
these graphs must satisfy in
order to be $\mathcal{F}(1,2)$-invariant.  The potential candidates partition naturally into
three parts.  The candidates for which $b-a=2$ are honeycomb toroidal graphs and they are
treated in Section 5.  Section 6 deals with the special case of $a=1$ for which there are useful
blocks of imprimitivity.  

The remaining graphs of the form $\mathcal{M}_O(n,a,b)$ are discussed in Sections 7 and 8
where the focus is on the various types of vertex stabilizers which may occur.  Section 7 considers
those which are arc-transitive and Section 8 examines those which are not arc-transitive.
There are many cases that arise in Sections 7 and 8.  We give details for some of the cases
leaving cases for which the details are just analogues to the reader.

\section{Generalized Petersen Graphs}

The rest of the paper is devoted to the proof of Theorem \ref{main}.  We start by looking at
how generalized Petersen graphs become involved. 

\begin{lemma}\label{dh}  Let $X$ be a connected trivalent vertex-transitive $\mathcal{F}(1,2)$-invari\\ant graph
of order $2n$, where the 2-factor $F$ consists of two $n$-cycles $C_1$ and $C_2$.  If $H$ is the
restriction of the setwise stabilizer $\mathrm{Aut}(X)_{\{U\}}$ to $U$, 
then $H\leq D_n$.
\end{lemma}
\begin{proof} Because $X$ is $\mathcal{F}(1,2)$-invariant, we know that $H$ must act transitively on
$U$.  The automorphism group of a cycle of length $n$ is the dihedral
group $D_n$ from which the result follows.      \end{proof}.

\begin{thm}\label{gp} Let $X$ be a connected trivalent vertex-transitive $\mathcal{F}(1,2)$-invariant graph
of order $2n$, where the 2-factor $F$ consists of two $n$-cycles $C_1$ and $C_2$.  Let $H$ be the
restriction to $U$ of the setwise stabilizer $\mathrm{Aut}(X)_{\{U\}}$.  If $H$ contains an
$n$-cycle, then $X$ is a generalized Petersen graph.
\end{thm} 
\begin{proof} If $H$ contains an $n$-cycle, then $\rho=(u_0\;u_1\;\cdots\;u_{n-1})\in H$.
Then the faithful extension of $\rho$ is a product of two $n$-cycles.
This implies the subgraph induced on $V$ is a circulant graph of valency 2, that is, there is a
$k$ such that the edges have the form $[v_i,v_{i+k}]$.  This is the generalized Petersen graph
$\mathrm{GP}(n,k)$.  Note that because this induced subgraph is an $n$-cycle,
we have $\mathrm{gcd}(n,k)=1$.  
\end{proof}

\bigskip

The automorphisms of generalized Petersen graphs were determined in \cite{F1}.  Thus, we know
precisely which generalized Petersen graphs are $\mathcal{F}(1,2)$-invariant.  They are precisely
those listed in the first item describing the collection $\mathcal{C}$.

\section{Another Potential Family}

One way of looking at generalized Petersen graphs is that they arise naturally by considering the
action of the faithful extension of $D_n$ acting on two sets where the desired outcome is a
connected trivalent graph composed of two 2-regular graphs, one of which is a cycle, joined by
a 1-factor.  So it was obvious to have a permutation which is a product of two $n$-cycles acting
as an automorphism on the graph, that is, the restriction to one set is an $n$-cycle.

However, when $n$ is even, there is a transitive subgroup of $D_n$ that does not contain an
$n$-cycle \cite{C2}.  Namely, it is the subgroup $\langle\rho^2,\rho\tau\rangle$.
We now describe a family $\mathcal{M}$ of trivalent graphs which admit the faithful extension of
the group $\langle\rho^2,\rho\tau\rangle$ as automorphisms and have the edges in $C_1$ and the 1-factor
edges as defined earlier.  The edges on the vertices $V$ require definition.  We want the
graph defined by these edges to be a 2-factor.

There are two cases.  Suppose there is an edge from $v_i$ to $v_{i+k}$, where both $i$ and $k$
are even.  The action of the faithful extension of $\rho^2$ generates a 2-factor on the vertices 
$\{v_i:i\mbox{ even}\}$.
A 2-factor on the vertices with odd subscripts requires an even ``jump'' as well.  The edges are
then described with two parameters $a,b$, both even, so that there are the edges $[v_i,v_{i+a}]$
when $i$ is even, and the edges $[v_i,v_{i+b}]$ when $i$ is odd.  The notation for this graph
is $\mathcal{M}_E(n,a,b)$. 

We have seen that if one of the parameters describing the edges of the subgraph induced on $V$
is even, then both must be even.
This leaves the case that both parameters are odd.  In this case the graph $\mathcal{M}_O(n,a,b)$
has the edges $[v_i,v_{i+a}]\mbox{ and }[v_i,v_{i+b}]$ for all even $i$.  These graphs form
a subfamily of the $GDGP$ graphs defined in \cite{J1}; in particular, $\mathcal{M}_O(n,a,b)$
is the same as $GDGP_2(n,a,n-b)$.

Note that $\mathcal{M}_E(n,a,b)$ is a generalized Petersen graph if $a=b$.  Similarly,
$\mathcal{M}_O(n,a,b)$ is a generalized Petersen graph when $a=n-b$.  Also
note that the subgraph induced on $V$ is never connected for any $\mathcal{M}_E(n,a,b)$.
The situation for graphs with both parameters odd is captured in the following result.

\begin{lemma}\label{span} The subgraph induced on $V$ in $\mathcal{M}_O(n,a,b)$ is a
spanning cycle if and only if $\mathrm{gcd}(b-a,n)=2$.
\end{lemma}
\begin{proof}Without loss of generality let $a<b$.  Start tracing the 2-factor at $v_a$ with the
edge $[v_a,v_0]$.  It continues with the 2-path $[v_a,v_0,v_b]$.  The 2-factor alternates 
odd and even subscripted vertices and covers a jump of length $b-a$ in going from one odd
subscripted vertex to the next.  This means it is jumping $(b-a)/2$ successive odd subscripted
vertices.  As there are $n/2$ odd subscripted vertices, the 2-factor passes through all of them
if and only if $\mathrm{gcd}((b-a)/2,n/2)=1$.   This is equivalent to $\mathrm{gcd}(b-a,n)=2$.
\end{proof}

\begin{lemma}\label{reg}If $X=\mathcal{M}_O(n,a,b)$ is $\mathcal{F}(1,2)$-invariant and
$n-b\neq a$, then $|\mathrm{Aut}(X)|=2n$ which means $\mathrm{Aut}(X)$ is a regular group.
\end{lemma}
\begin{proof} Consider an automorphism $h$ that fixes $u_0$.  Because $h$ must preserve the
1-factor, it also fixes $v_0$.  Thus, if $h$ is not the identity, it must interchange $u_1$ and $u_{n-1}$
which implies that $h$ is the faithful extension of $\tau$.  This in turn implies that the
faithful extension of $\rho\in\mathrm{Aut}(X)$ which
contradicts the hypothesis that $n-b\neq a$.     \end{proof}

\medskip

\begin{thm}\label{vt} Let $X=\mathcal{M}_O(n,a,b)$ satisfy $\mathrm{gcd}(b-a,n)=2$
and $a\neq n-b$.  The automorphism group of $X$ contains a regular subgroup $G$ which
preserves the partition of $E(X)$ into the two $n$-cycles on $U$ and $V$, respectively, and
the perfect matching joining them if and only if $(b-a)^2/2\equiv 2(\mbox{\rm mod }n)$
and one of $a+(a-1)(a-b)/2\equiv 1(\mbox{\rm mod }n)$ or
$b+(b-1)(b-a)/2\equiv 1(\mbox{\rm mod }n)$ holds.
\end{thm} 
\begin{proof} By Lemma \ref{span} we know that the subgraph induced on $V$ is an $n$-cycle.
Label the vertices of the cycles as before.  Assume that $\mathrm{Aut}(X)$ contains a regular
subgroup $G$ preserving the partition of $E(X)$ as stated.  Let $H$ be the subgroup of $G$
that maps $U$ to $U$ (and thus $V$ to $V$).  The faithful extension of $\rho$ does
not belong to $H$ because $a\neq n-b$.  Therefore, $H$ is generated by the faithful extensions
of $\rho^2$ and $\rho\tau$.

Let $f\in G$ be the automorphism that maps $u_0$ to $v_0$.  It follows that $f(v_0)=u_0$
because $f$ preserves the 1-factor.  Then $f$ is an involution because $f^2$ fixes $u_0$
and $G$ is regular.  We conclude that $f$ is a product of transpositions because it must 
interchange $U$ and $V$. 

The value of $f(u_1)$ completely determines $f$ because it maps $C_1$ onto $C_2$ and is
a product of transpositions.  There are two possibilities:  Either $f(u_1)=v_a$ or $f(u_1)=v_b$.
We first examine the case that $f(u_1)=v_a$.

We have that $f(u_2)=v_{a-b}$ because $f(u_1)=v_a$.  This implies that the transposition
$(u_2\;\;v_{a-b})$ is in $f$.  Continuing in this way, we see that $f(u_i)=v_{i(a-b)/2}$ for all
even $i$.  Thus, $f(u_{a-b})=v_{(a-b)^2/2}$.  Because $f$ must preserve the 1-factor and the
transposition $(u_2\;\;v_{a-b})$ is in $f$, $(a-b)^2/2\equiv 2(\mbox{mod }n)$ must hold.

Because $f(u_1)=v_a$, using an argument analogous to that in the preceding paragraph, it
is easy to see that $f(u_i)=v_{a+(i-1)(a-b)/2}$ for $i$ odd.  So $f(u_a)=v_{a+(a-1)(a-b)/2}$
which implies $a+(a-1)(a-b)/2\equiv 1(\mbox{mod }n)$ because the 1-factor is preserved. 
This establishes the two conditions of the conclusion.

Now suppose the two conditions hold.  We know that the faithful extensions of $\rho^2$ and
$\rho\tau$ are automorphisms of $X$ and the group they generate has $U$ and $V$
as the two orbits.  Define $f$ by $f(u_i)=v_{i(a-b)/2}$ for $i$ even, $f(u_i)=v_{a+(i-1)(a-b)/2}$
for $i$ odd, and $f$ is a product of transpositions with no fixed points.  It's easy to see that
$f$ interchanges $C_1$ and $C_2$ preserving the edges of both.  What we need to verify is that
$f$ preserves the 1-factor joining the two cycles.

By definition $(u_0\;v_0)$ is a transposition in $f$ preserving the edge $[u_0,v_0]$ of the
1-factor.  When $i$ is even, $f(u_i)=v_{i(a-b)/2}$.  Then $$f(u_{i(a-b)/2})=v_{i(a-b)^2/4}=
v_i$$ because $i(a-b)^2/4=\frac{i}{2}(a-b)^2/2\equiv i(\mbox{mod }n)$.  From this we see
that $f$ preserves the edges of the 1-factor whose end vertices have even subscripts.

When $i$ is odd, $f(u_i)=v_{a+(i-1)(a-b)/2}$.  Then $$f(u_{a+(i-1)(a-b)/2})=v_{a+
(a-1)(a-b)/2+(i-1)(a-b)^2/4}=v_i$$ because $a+(a-1)(a-b)/2\equiv 1(\mbox{mod }n)$
and $(i-1)(a-b)^2/4\equiv i-1(\mbox{mod }n)$.  Hence, $f$ preserves the edges of the 1-factor
whose end vertices have odd subscripts.  

Therefore, the group $G$ generated by $f$ and the faithful extensions of $\rho^2$
and $\rho\tau$ is a regular subgroup of $\mathrm{Aut}(X)$ preserving the edge
partition of $X$ as claimed.

In the case that $f(u_1)=v_b$, then the analogue of the preceding argument is valid and
we end up with the conguences $(b-a)^2/2\equiv 2(\mbox{mod }n)$ and
$b+(b-1)(b-a)/2\equiv 1(\mbox{mod }n)$ being forced to hold.  If we now define a function
$g$ by $g(u_i)=v_{i(b-a)/2}$ for $i$ even, $g(u_i)=v_{b+(i-1)(b-a)/2}$ for $i$ odd and $g$
is a product of disjoint transpositions, then the group $G$ generated by $g$ and the faithful
extensions of $\rho^2$ and $\rho\tau$
is the subgroup of $\mathrm{Aut}(X)$ for which we are looking.    \end{proof}

\medskip

It should be noted that not both conditions at the end of the statement of Theorem \ref{vt}
may hold.  It can be shown that if both conditions hold, then $a=n-b$ which violates one of
the hypotheses.

Theorem \ref{vt} provides us with another family of graphs which may contain $\mathcal{F}(1,2)$-invariant
graphs.  We now explore this family and shall discover several subcases that arise
because of the congruence conditions in Theorem \ref{vt}

\section{Honeycomb Toroidal Graphs} 

The next result tells us that some of the graphs arising in the family from Section 4 are, in
fact, honeycomb toroidal graphs.  This subfamily is treated separately in this section.
   
\begin{thm}\label{htg} The graph $\mathcal{M}_O(n,a,a+2)$ is isomorphic to $\mathrm{HTG}
(2,n,a+1)$.
\end{thm} 
\begin{proof} We are going to define an isomorphism from the vertex set of $\mathcal{M}_O(n,a,a+2)$
to the vertex set of $\mathrm{HTG}(2,n,a+1)$, and even though we have used the same labels for
both sets of vertices, no confusion should arise as the domain and range have been
clearly specified.  We define $g(u_i)=u_{i+1},g(v_j)=v_{j+1}, j\mbox{ even and }g(v_j)=
v_{j-a}, j\mbox{ odd}$, where the subscripts are calculated modulo $n$. It is straightforward
to verify that $g$ is an isomorphism as claimed.     \end{proof}

\medskip

The fourth author has determined the automorphism groups of honeycomb toroidal graphs in
\cite{S1}.  This allows us to determine which graphs $\mathcal{M}_O(n,a,a+2)$ are
$\mathcal{F}(1,2)$-invariant.  We present the results in terms of $\mathrm{HTG}(2,n,\ell)$
as this is the description of the graphs we prefer.  Note that $\ell=a+1$ according to
Lemma \ref{htg}.

Honeycomb toroidal graphs are Cayley graphs on generalized dihedral groups \cite{A2}.
Recall that a Cayley graph $X$ on a group $G$ is {\it normal} when the left-regular representation
of $G$ is a normal subgroup of $\mathrm{Aut}(X)$.  It is shown in \cite{S1}
that there are only a few non-normal honeycomb toroidal graphs of the form
$\mathrm{HTG}(2,n,\ell)$ and none of them is $\mathcal{F}(1,2)$-invariant.  They are
$\mathrm{HTG}(2,4,0)$ (which is isomorphic to $\mathrm{HTG}(2,4,2)$), $\mathrm{HTG}(2,8,4)$
and those of the form $\mathrm{HTG}(2,2a,2)$, $a>2$.

It is easy to see that $\ell>0$ and is even because $a+1>0$ and $a$ is odd.  When
$\mathrm{HTG}(2,n,\ell)$ is a normal Cayley graph, the stabilizer of a vertex is a subgroup
of $S_3$, the symmetric group of degree 3, and four conditions are presented in \cite{S1}
which depend on the stabilizer of a vertex.  
 
The conditions are:
\begin{enumerate}
\item $\mathrm{gcd}(n,\ell+2)=4$ and $4n|(\ell^2+4\ell-12)$;
\item $\mathrm{gcd}(n,\ell-2)=4$ and $4n|(\ell^2-4\ell-12)$;
\item $\ell=n/2$; and
\item $\mathrm{gcd}(n,\ell+2)=4=\mathrm{gcd}(n,\ell-2)$ and $4n|(\ell^2+12)$.
\end{enumerate}  It is then proven in \cite{S1} that if none of the four conditions is satisfied,
then $\mathrm{HTG}(2,n,\ell)$ is $\mathcal{F}(1,2)$-invariant.  If at least two of the
conditions are satisfied, then all four conditions are satisfied and $\mathrm{HTG}(2,n,\ell)$
is not $\mathcal{F}(1,2)$-invariant because the stabilizer contains a 3-cycle.  The latter
conclusion also holds if condition 4 is satisfied.

Finally, if precisely one of the first three conditions holds, then the stabilizer of a vertex has order 2.
Two of the stabilizers do not leave the partition invariant and the only one that does is $\ell=n/2$
(and so $\ell$ being even implies $n\equiv 0(\mbox{mod }4)$).  Thus, $\mathrm{HTG}(2,n,n/2)$
is $\mathcal{F}(1,2)$-invariant for $n\equiv 0(\mbox{mod }4)$ and $n>8$ (recall that
$\mathrm{HTG}(2,8,4)$ is not normal).  However, the latter graph is isomorphic to
the generalized Petersen graph $\mathrm{GP}(n,n/2-1)$.

As we are interested in the graphs of the form $\mathrm{HTG}(2,n,\ell)$ that are not generalized
Petersen graphs, we eliminate $\ell=0$ and $\ell=n/2$.  We also remind the reader that
$\mathrm{HTG}(2,n,\ell)$ and $\mathrm{HTG}(2,n,n-\ell)$ are isomorphic.

Because we have eliminated $\ell=n/2$, we discard condition 3 above and we have seen
that $\mathrm{HTG}(2,n,\pm\ell)$, $0<\ell<n/2$, is $\mathcal{F}(1,2)$-invariant if and only
if none of conditions 1, 2 or 4 are satisfied.  This is the second item in the membership list
for $\mathcal{C}$.

\section{A Sparse Class---Small Girth}

The preceding section dealt with the special case $b-a=2$ for which the conditions $\mathrm{gcd}
(b-a,n)=2, (b-a)^2/2\equiv 2(\mbox{mod }n)\mbox{ and }a+(a-1)(a-b)/2\equiv 1(\mbox{mod }n)$
of Theorem \ref{vt} are trivially satisfied.   Thus, there are
many candidates for $\mathcal{F}(1,2)$-invariant graphs when $b-a=2$.  
However, when $b-a>2$, the number of possible graphs decreases dramatically.
Two examples illustrating this follow.  When $b-a=4$, we obtain $8\equiv 2(\mbox{mod }n)$
which implies that $n=6$.  The only possibility is $\mathcal{M}_O(6,1,5)$ which is the
generalized Petersen graph $\mathrm{GP}(6,1)$.  When $b-a=6$, we obtain
$18\equiv 2(\mbox{mod }n)$ so that $n=8$ or $n=16$.  When $n=8$, the only possibility is
the generalized Petersen graph $\mathrm{GP}(8,1)$.  When $n=16$, there are two
non-isomorphic possibilities: $\mathcal{M}_O(16,1,7)$ and $\mathcal{M}_O(16,3,9)$.  This
strongly indicates why we are calling this class a sparse class.

We now give a definition which simplifies many subsequent statements.
\begin{defn}\label{feasi}{\em The graph $X=\mathcal{M}_O(n,a,b)$ is defined to be}
feasible {\em if it satisfies Theorem \ref{vt} and the following inequalities:}
\begin{equation} 1\leq a<b-2<n-a-2.
\end{equation}
\end{defn} 

Note that inequality (1) implies that $b-a>2$ which takes advantage of the results of Section 5.  
Because $b=n-a$ means that $X$ is a generalized Petersen graph, we need not consider this situation
and inequality (1) reflects that exclusion.
Also note that $\mathcal{M}_O(n,a,b)$ and $\mathcal{M}_O(n,n-b,n-a)$ are isomorphic.
Inequality (1) implies that $a<n-b$ so that we are examining the isomorph with the smaller of
the two minimum odd jumps.  Finally, inequality (1) implies that $a<n/2$.

Let $G$ be the regular subgroup of $\mathrm{Aut}(X)$ in the statement of Theorem \ref{vt}.
The group $G$ has three orbits acting on the edges of $X$.  One orbit consists of the 1-factor joining
the two $n$-cycles and we color these edges red.  A second orbit contains the edge
$[u_0,u_1]$ and we color these edges blue.  The final orbit contains the edge
$[u_1,u_2]$ and we color these edges green.  We denote these orbits with $\mathcal{R},
\mathcal{B}\mbox{ and }\mathcal{G}$, respectively.  Note that every vertex is incident with one
edge of each color.  We frequently refer to this coloring in the ensuing discussion.

There are a few useful facts about $\mathcal{R},\mathcal{B}\mbox{ and }\mathcal{G}$ that
are encapsulated in the following lemma.  This result sets the stage for the cases arising in the
completion of the proof of the main theorem.

\begin{lemma}\label{orb} If $X=\mathcal{M}_O(n,a,b)$ is feasible and $\mathcal{R},
\mathcal{B}\mbox{ and }\mathcal{G}$ are as described above,
then the edge orbits of $\mathrm{Aut}(X)$ are one of the following:\\
\indent{\em(i)} $\mathcal{R},\mathcal{B}\mbox{ and }\mathcal{G}$;\\
\indent{\em(ii)} $\mathcal{R}\cup\mathcal{B}\mbox{ and }\mathcal{G}$;\\
\indent{\em(iii)} $\mathcal{R}\cup\mathcal{G}\mbox{ and }\mathcal{B}$; or\\
\indent{\em(iv)} $\mathcal{R}\cup\mathcal{B}\cup\mathcal{G}$ in which case $X$ is arc-transitive.
\end{lemma}
\begin{proof} Because the orbits of a permutation group are defined as the equivalence classes
of an equivalence relation and $\mathcal{R},\mathcal{B}\mbox{ and }\mathcal{G}$ are the
edge orbits of $G$, it follows that the orbits of $\mathrm{Aut}(X)$ are a disjoint collection
of unions of $\mathcal{R},\mathcal{B}\mbox{ and }\mathcal{G}$.  It suffices to show that
$\mathcal{B}\cup\mathcal{G}$ and $\mathcal{R}$ may not occur. 

If the edge orbits of $\mathrm{Aut}(X)$ are $\mathcal{B}\cup\mathcal{G}$ and $\mathcal{R}$,
then there must be an automorphism mapping the edge $[u_0,u_1]$ onto the edge $[u_{n-1},u_0]$.
Both ways this may happen force $b=n-a$ which violates $X$ being feasible.

Finally, if condition (iv) happens, then Tutte \cite{T1} proved that a trivalent graph that is both
vertex-transitive and edge-transitive also is arc-transitive.    \end{proof}  

\medskip

We break the considerations for this class of graphs into four cases.  It is easy to see that
$\mathcal{M}_O(n,a,b)$ has girth 4 if and only if $a=1$ or $b-a=n/2$.  We need not consider
the latter case because $\mathrm{gcd}(n,n/2)=2$ implies that $n=4$ forcing the graph to be
$\mathrm{GP}(4,1)$.  Thus, we consider the case of $a=1$
in this section.  Moreover, the 4-cycles of $\mathcal{M}_O(1,b)$ form a 2-factor.   The next
result describes some useful block systems for $\mathrm{Aut}(\mathcal{M}_O(1,b))$, but a
definition is required first.  Note that the edge $[v_0,v_1]$ is blue because $a=1$. 

We introduce an auxiliary graph $Y$ to describe the first block system.
Let $Y$ have the same vertex set as $X$.   The edges of $Y$ are the edges of $\mathcal{G}$
together with the diameter edges of the 4-cycles, that is, given the 4-cycle $[u_i,u_{i+1},
v_{i+1},v_i,u_i]$ in $X$, $i$ even, then $[u_i,v_{i+1}]$ and $[u_{i+1},v_i]$ are edges of $Y$.  The
graph $Y$ is regular of valency 2 so that it consists of vertex-disjoint cycles.  Let {\bf A}
denote the collection of vertex sets of the cycles comprising $Y$.  The elements of {\bf A}
are called {\it anchor chains} of $X$.

Let {\bf F} denote the collection of vertex sets of the 4-cycles, that is, the vertex sets of the
subgraph composed of the edges of $\mathcal{R}\cup\mathcal{B}$.  The edges of
$\mathcal{G}$ form a perfect matching in $X$.  Let {\bf M} be the collection of 2-subsets of
ends of the edges in this perfect matching. 

\begin{lemma}\label{imp} If $X=\mathcal{M}_O(n,1,b)$, $b>3$, is feasible, then {\bf A}, {\bf F}
and {\bf M} are imprimitive block systems for $\mathrm{Aut}(X)$.  Moreover, if
$f\in\mathrm{Aut}(X)$ fixes a vertex of an anchor chain $\mathcal{A}$, then it fixes every
vertex of $\mathcal{A}$.
\end{lemma}
\begin{proof} The elements of {\bf F} form a block system because the
4-cycles are vertex-disjoint and are the only 4-cycles in $X$.  The edges of $\mathcal{G}$ belong to
no 4-cycles and form a perfect matching in $X$.  Thus, {\bf M} is an imprimitive block
system for $\mathrm{Aut}(X)$.

Suppose $f\in\mathrm{Aut}(X)$ fixes a vertex $v$ in an anchor chain $\mathcal{A}$ from
{\bf A}.  Each vertex of $\mathcal{A}$ is incident with a single green edge so that the
vertex $w$ at the other end of the green edge containing $v$ also is fixed.  The next vertex
of $\mathcal{A}$
following $w$ is the vertex diametrically opposed to $w$ in a 4-cycle.  Hence, it also is
fixed.  Clearly, we may continue working around $\mathcal{A}$ obtaining that each vertex
is fixed.  This completes the proof.   \end{proof}

\medskip

Anchor chains play a significant role in determining whether a feasible $\mathcal{M}(n,1,b)$
graph is $\mathcal{F}(1,2)$-invariant.  For example, if $X=\mathcal{M}_O(n,1,b)$ has a single
anchor chain containing every vertex, then $X$ is $\mathcal{F}(1,2)$-invariant because
Lemma \ref{imp} implies that $\mathrm{Aut}(X)=G$.  It is easy to verify this happens
to be the case for $b=5$.

Denote the block of {\bf F} containing $u_i$, $i$ even,
with $\mathcal{F}_i$.  Thus, the blocks of {\bf F} are $\mathcal{F}_0,\mathcal{F}_2,
\ldots,\mathcal{F}_{n-2}$.  
Let $H$ denote the stabilizer of $u_0$ in $\mathrm{Aut}(X)$.  The anchor chain $\mathcal{A}_0$
containing $u_0$ contains the subchain $u_0,v_1,v_{1-b},u_{2-b},u_{3-b}$.  Hence, $H$
fixes the blocks $\mathcal{F}_0,\mathcal{F}_{1-b}\mbox{ and }\mathcal{F}_{3-b}$
setwise.  Suppose that $H$ contains an automorphism $f$ interchanging $v_0$ and $u_1$.
This implies that the anchor chain $\mathcal{A}_1$ containing $v_0$ is distinct from
$\mathcal{A}_0$.

Then $\mathcal{A}_1$ has the subchain $v_0,u_1,u_2,v_3,v_{3-b}$.   The length of the
anchor chain is even because alternate edges are green.  Because
$f$ interchanges the consecutive vertices $v_0$ and $u_1$ and the block
$\mathcal{F}_{3-b}$ is fixed setwise, it must interchange $v_{3-b}$ and $u_{4-b}$.
From this we see that the anchor chain has length 8 and that $u_{4-b}=u_{b-2}$ by
going around $\mathcal{A}_1$ in the other direction.  Because $b\neq 3$, this implies
that $b=(n+6)/2$. Therefore, the only possible graphs of the form $\mathcal{M}_O(n,1,b)$
which are not $\mathcal{F}(1,2)$-invariant satisfy $b=(n+6)/2$.  

Because the anchor chains have length 8, $n\equiv 0(\mbox{mod }4)$.  On the other hand,
because $2=\mathrm{gcd}(n,b-1)=\mathrm{gcd}(n,n/2+2)$, we see
that $n$ is divisible by 8.  We know that $n>8$ because $b<n-1$.  For all such $n>8$,
the involution \[(u_1\;v_0)(u_2\;v_{m+3})
(u_3\;v_{m+2})(u_4\;v_5)(u_{m+1}\;v_m)(u_{m+2}\;v_3)(u_{m+3}\;v_2)(u_{m+4}\;v_{m+5})\]
is an automorphism of $\mathcal{M}_O(n,1,b)$ in the stabilizer of $u_0$, where $m=n/2$.
Therefore, the graph is not $\mathcal{F}(1,2)$-invariant.  We have proved the following
result.  These graphs give the third item describing the family $\mathcal{C}$. 

\begin{thm}\label{girth4} Let $3<b<n-2$.  The feasible graph $\mathcal{M}_O(n,1,b)$ is
$\mathcal{F}(1,2)$-invariant if and only if the parameters do not satisfy $n\equiv
0(\mbox{\rm{mod} }8)$ and $b=(n+6)/2$.
\end{thm}

\section{A Sparse Class---Arc-Transitive}

The girth 4 and $b-a=2$ cases are completely settled in the preceding two sections so that we now
assume $a\geq 3$ and $\mathcal{M}_O(n,a,b)$ is feasible.  We continue to use the group $G$
and the edge coloring from the preceding section.  The essential approach is
that we consider feasible $\mathcal{M}_O(n,a,b)$-graphs and examine how their automorphism
groups can be proper supergroups of $G$.  Because we are assuming $b<n-a$, Lemma \ref{reg}
implies such a graph is $\mathcal{F}(1,2)$-invariant if and only if the stabilizer of a vertex is the
identity group. Thus, we concentrate on examining the stabilizer of a vertex.  There are many cases
arising from this approach.

The case we examine in this section is when $X=\mathcal{M}_O(n,a,b)$ is feasible and arc-transitive.  
This corresponds to (iv) in Lemma \ref{orb}.
The index 2 cyclic subgroup $\langle\rho^2\rangle$ of the setwise stabilizer of $U$ has four orbits on
$V(X)$.  Thus, $X$ is a trivalent arc-transitive tetracirculant.  These graphs are classified in \cite{F3}.

The classification states that the graphs are certain cyclic covers of the 3-dimensional cube
$Q_3$ or among a list of seventeen exceptional graphs all of which are in the Foster census
\cite{B1}.  Using {\sc magma} \cite{B2} it turns out that $\mathcal{M}_O(16,3,9)$ is
the only graph of the seventeen exceptions isomorphic to a feasible $\mathcal{M}_O(n,a,b)$
graph.  Because this graph is arc-transitive, it is of course not $\mathcal{F}(1,2)$-invariant and thus
has to be excluded from the collection $\mathcal{C}$ from Theorem 2.1. However, as the parameters of this 
graph in fact satisfy item 5 in the last group on page 4, which is due to our results in Section 8, this
graph already will be excluded from $\mathcal{C}$ this way, and so there is no need to mention it separately.  
This is the first convoluted situation mentioned in the outline of the proof at the end of Section 2.

We are left with the cyclic covers of $Q_3$, the so-called $CQ(k,m)$ graphs in \cite{F2,F3}.
By results of \cite{F3}, these covers have the property that if the automorphism group is
arc-transitive, then the group projects along the covering projection.  The results of \cite{F2}
imply that the graphs are 1-regular, that is, the stabilizer of a vertex has order 3.  (The 2-regular
graphs $CQ(1,2), CQ(1,3)$ and $CQ(1,6)$ may be ignored because they are isomorphic
to $\mathcal{M}_O(8,3,5),\\ \mathcal{M}_O(12,5,7)$ and $\mathcal{M}_O(24,5,19)$,
respectively.)  Thus, the two non-identity automorphisms of $X$ fixing $u_0$ cyclically permute the sets
$\mathcal{R},\mathcal{B}\mbox{ and }\mathcal{G}$.

We distinguish two cases depending on whether the edge $[v_0,v_a]$ is blue or green.
We examine one of the cases in detail and leave the other case to the reader.
Assume that $\alpha$ is an automorphism of $X$ fixing $u_0$ and mapping $\mathcal{B}$ to
$\mathcal{R}$, $\mathcal{R}$ to $\mathcal{G}$ and $\mathcal{G}$ back to $\mathcal{B}$.
Note that this implies that $\alpha$ maps the blue-green cycle containing $u_0$ to the blue-red
cycle containing $u_0$ which, in turn, is mapped to the red-green cycle containing $u_0$.
Hence, these three cycles have the same length.  Because the blue-green cycle has length $n$,
all three cycles have length $n$.  This completely determines the action of $\alpha$ enabling us
to determine necessary and sufficient conditions for its existence.

\smallskip

Assume that $[v_0,v_a]\in\mathcal{B}$.  The edge $[v_0,v_a]\in\mathcal{B}$ implies that the
blue-red cycle at $u_0$ contains \[ [\ldots,u_1,u_0,
v_0,v_a,u_a,u_{a-1},v_{a-1},v_{2a-1},u_{2a-1},u_{2(a-1)},\ldots],\] while the red-green cycle at $u_0$
contains \[ [\ldots,u_{-1},u_0,v_0,v_b,u_b,u_{b+1},v_{b+1},v_{2b+1},u_{2b+1},u_{2(b+1)},\ldots].\]
It is easy to see that these two cycles have length $n$ if and only if $4|n$ and
$\mathrm{gcd}(n,a-1)=4=\mathrm{gcd}(n,b+1)$.

Theorem \ref{vt} and $[v_0,v_a]\in\mathcal{B}$ imply that $(a-1)(a-b+2)/2\equiv 0(\mbox{mod }n)$.
Then $\mathrm{gcd}(n,a-1)=4$ implies that $2(a-b+2)=0$ in $\mathbb{Z}_n$.  It follows that
$b=n/2+a+2$ because $b\neq a+2$.  From the assumption that $\mathrm{gcd}(n,b-a)=2$, we
see that $\mathrm{gcd}(n,n/2+2)=2$.  The latter implies that $8|n$.

Write $a=4a_0+1$.  Then $4=\mathrm{gcd}(n,a-1)=\mathrm{gcd}(n,4a_0)$ implies
$\mathrm{gcd}(n,a_0)=1$.  In particular, $a_0$ is odd.  But then $$4=\mathrm{gcd}(n,b+1)=
\mathrm{gcd}(n,n/2+a+3)=\mathrm{gcd}(n,n/2+4(a_0+1))$$ implies that $n/2$ is not
divisible by 8.  Therefore, $n\equiv 8(\mbox{mod }16)$.

It is easy to verify that for each $n\equiv 8(\mbox{mod }16)$, each $a=4a_0+1$ with $\mathrm{gcd}
(n,a_0)=1$ and $a<n/4-1$, and $b=n/2+a+2$, the three parameters $n,a\mbox{ and }b$ satisfy
the conditions of Theorem \ref{vt}.  We now wish to determine for which values of the three
parameters an automorphism $\alpha$ of the form described above exists.  If it does exist, the
remarks above imply that for each $0\leq i<n/4$ and each $0\leq r<4$ we have

\[\alpha(u_{4i+r})=\left\{\begin{array}{ll}
u_{i(a-1)} & \mbox{when $r=0$}\\v_{i(a-1)} & \mbox{when $r=1$}\\
v_{i(a-1)+a} & \mbox{when $r=2$}\\u_{i(a-1)+a} & \mbox{when $r=3$}
\end{array} \right.\] 

and

\[\alpha(v_{4i+r})=\left\{\begin{array}{ll}
u_{i(a-1)-1} & \mbox{when $r=0$}\\v_{i(a-1)+b} & \mbox{when $r=1$}\\
v_{i(a-1)+a-b} & \mbox{when $r=2$}\\u_{i(a-1)+a+1} & \mbox{when $r=3$.}
\end{array} \right.\] 

The assumptions that $\mathrm{gcd}(n,a-1)=4$ and $\mathrm{gcd}(n,b-a)=2$ imply that
the above defined $\alpha$ is a permutation of the vertex set of $X$.  Moreover, it is clear
that $\alpha$ maps blue edges between vertices of $U$ to red edges, green edges between
vertices of $U$ to blue edges, and all red edges of $X$ to green edges of $X$.  It thus follows
that $\alpha$ is an automorphism of $X$ with the required properties regarding its effect on
$\mathcal{R},\mathcal{B}\mbox{ and }\mathcal{G}$ if and only if it maps all of the blue
edges between the $v_i$ vertices to red edges and maps all the green edges between the
$v_i$ vertices to blue edges.

Clearly it suffices to consider all the edges of the form $$[v_{4i},v_{4i+a}], [v_{4i},v_{4i+b}],
[v_{4i+2},v_{4i+2+a}]\mbox{ and }[v_{4i+2},v_{4i+2+b}].$$  Because $b=n/2+a+2=
4(n/8+a_0)+3$, it follows that $\alpha(v_{4i})=u_{i(a-1)-1},\\ \alpha(v_{4i+a})=v_{(i+a_0)(a-1)+b}
\mbox{ and }\alpha(v_{4i+b})=u_{(i+n/8+a_0)(a-1)+a+1}.$  Hence, $\alpha$ has the required
properties with respect to these edges if and only if $a_0(a-1)+b+1$ and $(n/8+a_0)(a-1)+a+3$
equal 0 in $\mathbb{Z}_n$.  In fact, the two conditions are equivalent because $\mathrm{gcd}
(n,a-1)=4$ and $b=n/2+a+2.$

Similarly, $\alpha(v_{4i+2})=v_{i(a-1)+a-b},\alpha(v_{4i+a+2})=u((i+a_0)(a-1)+a+1$
and $\alpha(v_{4i+b+2})=v_{(i+n/8+a_0+1)(a-1)+b}$, and so the condition for $\alpha$
to have the required properties with respect to these edges again is $a_0(a-1)+b+1=0$ and 
$(n/8+a_0+1)(a-1)+b=2a-b$ in $\mathbb{Z}_n$.   However, as $2(b-a)=4$ in
$\mathbb{Z}_n$, the latter condition is again equivalent to $a_0(a-1)+b+1=0$. Therefore,
$\alpha$ is an automorphism if and only if $n\equiv 8(\mbox{mod }16),a=4a_0+1$ for
$a_0$ such that $\mathrm{gcd}(n,a_0)=1$, and setting $b=n/2+a+2$ we have $\mathrm{gcd}
(n,b+1)=4$ and $a_0(a-1)+b+1=0$ in $\mathbb{Z}_n$. 

 Observing that $\mathrm{gcd}(n,b+1)=\mathrm{gcd}(n,n/2+4(a_0+1))$, we see that
$\mathrm{gcd}(n,b+1)=4$ if and only if no odd prime divisor of $n$ divides $a_0+1$ because
$a_0$ is odd and $n/2\equiv 4(\mbox{mod }8)$.  Finally, because \[a_0(a-1)+b+1=
4(a_0^2+a_0+1)+n/2, a_0^2+a_0+1\mbox{ is odd and }8|n,\] we have that
$a_0(a-1)+b+1=0$ in $\mathbb{Z}_n$ if and only if $8(a_0^2+a_0+1)=0$ in $\mathbb{Z}_n$.
Observe that the latter condition automatically implies $\mathrm{gcd}(n,a_0)=1$
(assuming $a_0$ is odd) and that no odd prime divisor of $n$ divides $a_0+1$.  This proves
the following result.

\begin{prop}\label{at1} Let $X=\mathcal{M}_O(n,a,b)$ be feasible and satisfy $a>1$.
If the edge $[v_0,v_a]$
is in $\mathcal{B}$, then $X$ admits an automorphism fixing $u_0$ and cyclically permuting
the sets $\mathcal{R},\mathcal{B}\mbox{ and }\mathcal{G}$ if and only if $n\equiv 8(\mbox{\rm{mod} }16),
a=4a_0+1<n/4-1$, where $a_0$ is odd and satisfies $8(a_0^2+a_0+1)=0$ in $\mathbb{Z}_n$,
and $b=n/2+a+2$.
\end{prop}

\smallskip

The other case is $[v_0,v_a]\in\mathcal{G}$.  We leave it to the reader to verify that the preceding
argument with $b$ replacing $a$ leads to the following analogue of Proposition \ref{at1}.

\begin{prop}\label{at2} Let $X=\mathcal{M}_O(n,a,b)$ be feasible and satisfy $a>1$.
If the edge $[v_0,v_a]$
is in $\mathcal{G}$, then $X$ admits an automorphism fixing $u_0$ and cyclically permuting
the sets $\mathcal{R},\mathcal{B}\mbox{ and }\mathcal{G}$ if and only if $n\equiv 8(\mbox{\rm{mod} }16),
b=4b_0+1$ with $n/2<b<(3n-4)/4$, where $b_0$ is odd and satisfies $8(b_0^2+b_0+1)=0$ in $\mathbb{Z}_n$,
and $a=b-n/2+2$.
\end{prop}

The graphs described in the two preceding propositions form the next entry for the family
$\mathcal{C}$ in Section 2.  We conclude this section with a comment about $\mathcal{M}_O(16,3,9)$.
As mentioned earlier, this graph is arc-transitive but it does not satisfy either Proposition \ref{at1}
or Proposition \ref{at2} as $n$ is divisble by 16. However, there is nothing wrong with this as this
graph in fact does not admit an automorphism cyclically permuting the sets $\mathcal{R},
\mathcal{B}\mbox{ and }\mathcal{G}$.  By the results in \cite{F1,F2} this is simply the only
arc-transitive $\mathcal{M}_O(n, a, b)$ graph which is not 1-regular and is neither a
generalized Petersen graph nor a graph of the form $\mathrm{HTG}(2,n,\ell)$. 

\section{Non-Trivial Stabilizer}

The preceding section takes care of the arc-transitive case, that is, case (iv) of Lemma \ref{orb}.
Consider case (i) of Lemma \ref{orb}, that is, when $\mathcal{R},\mathcal{B}\mbox{ and }
\mathcal{G}$ are distinct orbits of the action of $\mathrm{Aut}(X)$ on the edges of $X$.
This implies that the stabilizer of any vertex of $X$ must fix its three neighbors.  From this
it easily follows that the stabilizer of any vertex is the identity and
$X$ is $\mathcal{F}(1,2)$-invariant.

From the preceding paragraph, we see that if there are further non-$\mathcal{F}(1,2)$-invariant
examples, they must arise from cases (ii) and (iii) of Lemma \ref{orb}.  We consider case (ii) in
detail and leave case (iii) to the reader with a few hints thrown in.

We assume that $\mathrm{Aut}(X)$ has two orbits acting on the edges of $X$, namely,
$\mathcal{R}\cup\mathcal{B}$ and $\mathcal{G}$.  The next lemma shows that 8-cycles are
important.

\begin{lemma}\label{8c} Let $X=\mathcal{M}_O(n,a,b)$ be feasible and $a>1$.
If $\mathcal{B}\cup\mathcal{R}$ is an orbit of $\mathrm{Aut}(X)$ acting on the edges and
the cycles comprising the subgraph $\mathcal{B}\cup\mathcal{R}$ are not 8-cycles, then
the stabilizer of any vertex either is the identity or has order 2.
\end{lemma}
\begin{proof} For simplicity we refer to the cycles comprising $\mathcal{B}\cup\mathcal{R}$ as
blue-red cycles.  It is obvious that if there is a single blue-red cycle, then the stabilizer of a
vertex has order at most 2 and the conclusion follows.  So we assume there are at least two
blue-red cycles.

We consider distinct blue-red cycles to be {\it adjacent} if there is at least one green edge
joining vertices of the two cycles.  Consider two adjacent blue-red cycles, say \[[\ldots,v_{2-x},
u_{2-x},u_{1-x},v_{1-x},v_1,u_1,u_0,v_0,v_x,u_x,u_{x-1},v_{x-1}\ldots]\] and
\[[\ldots,v_{-x},u_{-x},u_{-x-1},v_{-x-1},v_{-1},u_{-1},u_{-2},v_{-2},v_{x-2},u_{x-2},
u_{x-3},v_{x-3},\ldots],\] where $x$ is one of $a$ or $b$ depending on which of the
conditions from Theorem \ref{vt} holds.  

Note that every fourth vertex of the first cycle,
starting at $u_0$, is joined by a green edge to a vertex of the second cycle.  Hence, the
length of the blue-red cycle is a multiple of 4.  It cannot be 4 because the girth of $X$ is
strictly greater than 4.  If the length is 12 or more, then any automorphism fixing $u_0$
and its three neighbors, must fix all the vertices of the two blue-red cycles.  It then
must be the identity on all of $X$ because $X$ is connected.  The result follows.    \end{proof}

\medskip

Lemma \ref{8c} establishes two subcases: 1) when the stabilizer of $u_0$ contains a 
non-identity automorphism that fixes its three
neighbors and the length of the blue-red cycles is 8; and 2) when the stabilizer of $u_0$
has order 2 and its unique non-identity element interchanges $v_0$ and $u_1$.  

\medskip

We consider subcase 1) first and this produces the second convoluted situation.  Namely, we
establish the values the parameters must have if the assumptions of this subcase hold.  We then
show that, in fact, there is a non-identity automorphism in the stabilizer, although it does not fix
$u_0$ and its three neighbors, and the blue-red cycles have length 8.  This still establishes
that the corresponding $\mathcal{M}_O(n,a,b)$ is not $\mathcal{F}(1,2)$-invariant.
 
If $[v_0,v_a]$ is blue, then the blue-red cycle containing $u_0$ is \[[u_0,v_0,v_a,u_a,u_{a-1},
v_{a-1},v_{2a-1},u_{2a-1},u_0]\] which implies that $a=(n+2)/2$.    However, this contradicts
$1<a<b-2<n-a-2$ so that $[v_0,v_b]$ must be blue.  Looking at the same 8-cycle with $b$
replacing $a$, we obtain $b=(n+2)/2$ which implies $n$ is divisible by 4 because $b$ is odd.  

Let $\eta$ be a non-identity element of $\mathrm{Aut}(X)$ fixing $u_0$ and its three neighbors.
If $\eta$ also fixes $u_i$ and its three neighbors for every $u_i$, $i$ even, then $\eta$
would be the identity.  Hence, by relabelling the vertices if necessary, we may assume that
$\eta$ fixes $u_{n-1},u_0,v_0,u_1\mbox{ and }u_2$, and interchanges $v_2$ and $u_3$.
Let $g\in G$ be the automorphism switching $u_1$ and $u_2$.  Define $\theta=\eta g\eta g$. 
It is easy to verify that $\theta$ fixes $u_1$ and $u_2$ and contains
the two transpositions $(u_0\;v_1)$ and $(u_3\;v_2)$ when written as a product of disjoint 
cycles.

Consider the 10-cycle \[[u_1,u_2,v_2,v_{2+a},v_{2+a-b},u_{2+a-b},u_{1+a-b},v_{1+a-b},
v_{1-b},v_1,u_1].\]  The action of $\theta$ on this 10-cycle provides useful information.  It
fixes $u_1$ and $u_2$,  and switches $v_2$ and $u_3$.  Because $[v_2,v_{2+a}]$ is a green
edge, it must switch $v_{2+a}$ and the green neighbor of $u_3$ which is $u_4$.  Thus,
$v_{2+a-b}$ is mapped to either $v_4$ or $u_5$.  If $v_{2+a-b}$ is mapped to $v_4$,
then $u_{2+a-b}$ is mapped to $v_{4+b}$.  This implies $u_{1+a-b}$ is mapped to
$v_{4+b-a}$ because $[u_{2+a-b},u_{1+a-b}]$ is in $\mathcal{G}$. 

Looking at the cycle from the other end, we have $\theta(v_1)=u_0$ which implies that
$\theta(v_{1-b})=v_0$.  Because $[v_{1+a-b},v_{1-b}]\in\mathcal{G}$, we have
$\theta(v_{1+a-b})=v_a$.  Hence, the edge $[v_{4+b-a},v_a]$ is either blue or red.
As there are no red edges between vertices of $V$, it is blue so that $4+2b-2a=0$ which
implies $a=3$.  However, $\eta$ fixes $v_0$ and
switches $u_3$ and $v_2$.  When $a=3$, $[v_0,v_3]\in\mathcal{G}$ so that $\eta$
must fix $v_3$ as well, but this is impossible as $v_2$ and $v_3$ are not adjacent.  

Therefore, we conclude that $\theta(v_{2+a-b})=u_5$ so that $\theta(u_{2+a-b})=
v_5$ and $\theta(u_{1+a-b})=v_{5-a}$.  This implies $5-a+b\equiv a(\mbox{mod }n)$
which is equivalent to $2a\equiv n/2+6(\mbox{mod }n)$.  Because $a<n/2$ and is
odd, we have that $8|n$ and $a=n/4+3$. 

As $b-a=n/4-2$, the conditions $\mathrm{gcd}(b-a,n)=2$ and $(b-a)^2/2\equiv 2(\mbox{mod }n)$
imply $n\equiv 16(\mbox{mod }32)$.  If we choose $n,a\mbox{ and }b$ as just described, it
is straightforward to check that they satisfy Theorem \ref{vt} and (1).  Hence, what is left
to do is to determine exactly for which $n$ there is an automorphism in the stabilizer of $u_0$
fixing all three of of its neighbors.  We work with $\theta$. 

Denote the blue-red 8-cycle containing the edge $[u_i,u_{i+1}]$, $i$ even and $0\leq i<n/2$, by $B_i$.
Note that $\theta$ reflects $B_0$ and $B_2$ at the respective vertices $u_1$ and $u_2$.
Because the cycles have length 8, the vertices $u_{n/2+1}$ and $u_{n/2+2}$ also are fixed.

We saw above that $\theta$ switches $u_4$ and $v_{2+a}=v_{n/4+5}$ so that it switches the
8-cycles $B_4$ and $B_{1+a}$.  We also saw that $\theta(v_{3n/4+4})=u_5$, where
$v_{3n/4+4}\in B_{1+a}$ and $u_5\in B_4$.  Then $\theta$ maps $u_5$ to either 
$v_{3n/4+4}$ or $u_{2+a}=u_{n/4+5}$.  We know that $\theta$ switches $v_0$ and $v_{1-b}$
which implies it switches the green edges incident with them.  Thus, $\theta$
switches $v_a=v_{n/4+3}$ and $v_{1+a-b}=v_{3n/4+3}$.  These two vertices are antipodal
on $B_{n/4+2}$ which implies $\theta$ fixes $B_{n/4+2}$ setwise.  We then know that
$\theta$ reflects $B_{n/4+2}$ with respect to the two vertices $u_{n/4+2}$ and $u_{3n/4+2}$.
This implies that $\theta$ interchanges $v_{3n/4+2}$ and $u_{3n/4+3}$.  Then the green
edges incident with the two vertices are switched and they are $u_{3n/4+4}$ and $v_5$.  We
conclude that $\theta(v_{3n/4+4})=u_5$, that is, $\theta$ switches these
two vertices.

Then the respective green edges incident with $u_5$ and $v_{3n/4+4}$ are switched.
Thus, $u_6$ and $v_7$ are switched.  As $u_6$ and  $v_7$ belong to the same
8-cycle $B_6$, $\theta$ reflects $B_6$ and we see that $\theta$ fixes $u_7$. 

The green edge incident with $u_7$ is $[u_7,u_8]$ so that $\theta$ fixes $u_8$.  
We then have that $\theta$ fixes $B_8$ setwise and we want to
show that $\theta$ switches $v_8$ and $u_9$, that is, $\theta$ reflects $B_8$ with
respect to $u_8$.  If $\theta$ does not reflect $B_8$, then $\theta$ fixes every
vertex of $B_8$.  In particular, $v_{n/2+9}$ is fixed.  The green edge incident with
$v_{n/2+9}$ is $[v_{n/4+6},v_{n/2+9}]$ so that $v_{n/4+6}$ also is fixed.  
However, as $\theta$ maps $v_{a+2}=v_{n/4+5}$ to $u_4$ and $v_{a+2-b}=v_{3n/4+4}$
to $u_5$, it maps $u_{n/4+5}$ to $v_4$.  Hence, $\theta$ maps $u_{n/4+6}$ to
$v_{n/4+7}$ because $[u_{n/4+5},u_{n/4+6}]$ is green.  This is a contradiction because
$v_{n/4+6}$ and $v_{n/4+7}$ are not adjacent.  This means that $\theta$ switches $v_8$
and $u_9$.

We have shown that $u_8$ is fixed and $u_9$ and $v_8$ must switch.  Thus, we have
shown that starting with the conditions for $u_1$ and $u_2$ (that $\theta$ fixes both and
switches the two neighbors of each), the conditions are replicated
for $u_7$ and $u_8$.  Hence, $n$ is divisible by 6 and $\theta$ is completely
determined by the assumption of its action on $u_1$ and $u_2$ and their respective
neighbors.  In particular, for each $0\leq i<n/6$ and each $0\leq r<6$ we have that

\[\theta(u_{6i+r})=\left\{\begin{array}{ll}
v_{6i+r+1} & \mbox{when $r=0$}\\u_{6i+r} & \mbox{when $r=1,2$}\\
v_{6i+r-1} & \mbox{when $r=3$}\\v_{6i+r+\frac{n}{4}+1} & \mbox{when $r=4$}\\
v_{6i+r-\frac{n}{4}-1} & \mbox{when $r=5$}
\end{array} \right.\]  and

\[\theta(v_{6i+r})=\left\{\begin{array}{ll}
v_{6i+r+\frac{n}{2}} & \mbox{when $r=0,3$}\\u_{6i+r-1} & \mbox{when $r=1$}\\
u_{6i+r+1} & \mbox{when $r=2$}\\u_{6i+r+\frac{n}{4}+1} & \mbox{when $r=4$}\\
u_{6i+r-\frac{n}{4}-1} & \mbox{when $r=5$.}
\end{array} \right.\]

It can be verified that under the assumptions $n\equiv 16(\mbox{mod }32)$, $3|n$,
$a=n/4+3$ and $b=n/2+1$, the above mapping is an automorphism of $\mathcal{M}_O(n,a,b)$
which proves the following result.

\begin{prop} If $X=\mathcal{M}_O(n,a,b)$ is feasible, $a>1$, both $\mathcal{B}\cup
\mathcal{R}$ and $\mathcal{G}$ are orbits of $\mathrm{Aut}(X)$ acting on the edges,
and there is a non-identity automorphism of $X$ fixing $u_0$ and its three neighbors,
then $n\equiv 48(\mbox{\rm mod }96)$, $a=n/4+3$ and $b=n/2+1$.  Moreover, if
the parameters for $X$ are the preceding, then $X$ is not $\mathcal{F}(1,2)$-invariant.
\end{prop}

\medskip

We now consider the other subcase, that is, the stabilizer of $u_0$ does not have an element fixing
all the neighbors.  Note that this means that the stabilizer of $u_0$ has order 2 because if there
were two distinct automorphisms fixing $u_0$ and switching $u_1$ and $v_0$, then their
product would be a non-identity element fixing $u_0$ and all three neighbors.

Let $\sigma$ be the automorphism fixing $u_0$ and switching $u_1$ and $v_0$.  There are
two possibilities:  Either $[v_0,v_a]$ is blue or $[v_0,v_b]$ is blue.  Consider the former first. 

We know that $(a-1)(a-b+2)/2\equiv 0(\mbox{mod }n)$ from Theorem \ref{vt}.
The red edges of the blue-red cycle containing $[u_0,u_1]$ are swapped with the blue
edges of the cycle by $\sigma$.  In fact, $\sigma$ swaps the edges of $\mathcal{B}$ and
$\mathcal{R}$ as otherwise there is a non-identity automorphism fixing a vertex and each
of its neighbors.  
This implies it swaps the blue-green cycle of length $n$ containing $[u_0,u_1]$ with
the red-green cycle containing $[u_0,v_0]$.  The edges $[v_0,v_b]$ and $[u_b,u_{b+1}]$
are green in the latter cycle.  So the cycle contains \[[\ldots,u_{n-1},u_0,v_0,v_b,u_b,u_{b+1},
v_{b+1},v_{2b+1},u_{2b+1},\ldots].\]  Thus, its length is $4n/\mathrm{gcd}(b+1,n)=n$
which implies that $\mathrm{gcd}(b+1,n)=4$ and $4|n$.

Write the subscripts of the the vertices in the form $4i+r$ for $0\leq i<n/4$ and $0\leq r\leq 3$.
We then have the following:

\[\sigma(u_{4i+r})=\left\{\begin{array}{ll}
u_{i(b+1)} & \mbox{when $r=0$}\\v_{i(b+1)} & \mbox{when $r=1$}\\
v_{(i+1)(b+1)-1} & \mbox{when $r=2$}\\u_{(i+1)(b+1)-1} & \mbox{when $r=3$}
\end{array} \right.\] 

and

\[\sigma(v_{4i+r})=\left\{\begin{array}{ll}
u_{i(b+1)+1} & \mbox{when $r=0$}\\v_{i(b+1)+a} & \mbox{when $r=1$}\\
v_{(i+1)(b+1)-1-a} & \mbox{when $r=2$}\\u_{(i+1)(b+1)-2} & \mbox{when $r=3$.}
\end{array} \right.\] 

This mapping is bijective because $\mathrm{gcd}(b+1,n)=4$ and $\mathrm{gcd}(b-a,n)=2.$  
We just need to determine conditions so that all the edges between vertices of $V$ are mapped
to edges of $X$.  To this end, write \[a=4a_0+1\mbox{ and }b=4b_0+3.\] From above we have
\[\sigma(v_{4i+a})=v_{(i+a_0)(b+1)+a}, \sigma(v_{4i+b})=u_{(i+b_0+1)(b+1)-2}\mbox{ and }
\sigma(v_{4i})=u_{i(b+1)+1}\] so that $\sigma$ has the required properties with respect to
these corresponding edges if and only if \begin{eqnarray} a_0(b+1)+a-1\equiv 
0(\mbox{mod }n)\mbox{ and }(b_0+1)(b+1)\equiv 4(\mbox{mod }n).\end{eqnarray}

Similarly, $\sigma(v_{4i+2+a})=u_{(i+a_0+1)(b+1)-2},\sigma(v_{4i+2+b})=v_{(i+b_0+1)(b+1)+a}$
and $\sigma(v_{4i+2})=v_{(i+1)(b+1)-1-a}$ so that $\sigma$ has the required properties with
respect to these corresponding edges if and only if \begin{eqnarray} a_0(b+1)+a-1\equiv 0(\mbox{mod }n)
\mbox{ and }(b_0-1)(b+1)+2(a+1)\equiv 0(\mbox{mod }n).\end{eqnarray}

The second conditions of (2) and (3) together imply $2(b-a-2)\equiv 0(\mbox{mod }n)$.  This
holds if and only if $b=n/2+a+2$ because $b>a+2$.  We have $8|n$ because $\mathrm{gcd}
(b-a,n)=2$ and $4|n$.  Consequently, the second condition from (2) implies that $b_0$ is
even.  The condition $b<n-a$ implies $b<(3n+4)/4$.

Conversely, one can verify that, assuming $n$ is divisible by 8, $b=4b_0+3<(3n+4)/4$ with
$b_0>0$ even and $4(b_0+1)^2\equiv 4(\mbox{mod }n)$, and $a=b-n/2-2$,
the conditions of Theorem \ref{vt}, (1), (2), (3) and the assumptions we have made
hold so that the following result holds.

\begin{prop} If $\mathcal{M}_O(n,a,b)$ is feasible, $a>1$ and $[v_0,v_a]\in\mathcal{B}$, then
it admits an involutary automorphism swapping $\mathcal{B}$ and $\mathcal{R}$ if and only if
\begin{itemize}\item $8|n$,
\item $b=4b_0+3<(3n+4)/4$ with $b_0>0$ and even,
\item $4(b_0+1)^2\equiv 4(\mbox{\rm{mod} }n)$ and
\item $a=b-n/2-2$.
\end{itemize}
\end{prop}

Using an analogous argument, one can prove (left up to the reader) the following result for the
edge $[v_0,v_a]$ being green rather than blue.

\begin{prop} If $\mathcal{M}_O(n,a,b)$ is feasible, $a>1$ and $[v_0,v_a]\in\mathcal{G}$, then it admits
an involutary automorphism swapping $\mathcal{B}$ and $\mathcal{R}$ if and only if
\begin{itemize}\item $8|n$,
\item $a=4a_0+3<(n+4)/4$ with $a_0$ even,
\item $4(a_0+1)^2\equiv 4(\mbox{\rm{mod} }n)$ and
\item $b=n/2+a-2$.
\end{itemize}
\end{prop}

The preceding propositions cover all the subcases arising from part (ii) of Lemma \ref{orb}.
Part (iii) of Lemma \ref{orb} simply switches the roles of $\mathcal{B}$ and $\mathcal{G}$.
The arguments are analogues of those used in the above portion of Section 8 and we leave
them to the reader.  However, we did say we would provide a few hints and now fulfill that
promise.

There is an obvious analogue of Lemma \ref{orb} for which $\mathcal{G}\cup\mathcal{R}$
is substituted for $\mathcal{B}\cup\mathcal{R}$.  No hint is required for this as it is
a straightforward replacement argument.

The analogue of Proposition 8.2 is the third convoluted situation mentioned in the proof
outline, but the comments are essentially the same as those written for Proposition 8.2.
To obtain the analogue of Proposition 8.2, use the same $\eta$ but let $g\in G$ be the
automorphism switching $u_0$ and $u_1$ and examine the 10-cycle \[[u_0,u_1,v_1,v_{1-a},
v_{1+b-a},u_{1+b-a},u_{b-a},v_{b-a},v_b,v_0,u_0].\]  We then obtain the following result.

\begin{prop} If $X=\mathcal{M}_O(n,a,b)$ is feasible, $a>1$, both $\mathcal{R}\cup
\mathcal{G}$ and $\mathcal{B}$ are orbits of $\mathrm{Aut}(X)$ acting on the edges,
and there is a non-identity automorphism of $X$ fixing $u_0$ and its three neighbors,
then $n\equiv 48(\mbox{\rm mod }96)$, $a=n/4-3$ and $b=n/2-1$.  Moreover, if the
parameters for $X$ are the preceding, then $X$ is not $\mathcal{F}(1,2)$-invariant.
\end{prop}

The analogue of Proposition 8.3 is obtained by letting $\sigma$ switch $u_{n-1}$ and $v_0$
rather than $u_1$ and $v_0$.  The result is then the following.

\begin{prop} If $\mathcal{M}_O(n,a,b)$ is feasible, $a>1$ and $[v_0,v_a]\in\mathcal{B}$, then
it admits an involutary automorphism swapping $\mathcal{G}$ and $\mathcal{R}$ if and only if
\begin{itemize}\item $8|n$,
\item $a=4a_0+1<(n-4)/4$ with $a_0$ odd,
\item $4(a_0)^2\equiv 4(\mbox{\rm{mod} }n)$ and
\item $b=n/2+a+2$.
\end{itemize}
\end{prop}

Finally, the analogue of Proposition 8.4 is obtained by again letting $\sigma$ switch $u_{n-1}$ and $v_0$
rather than $u_1$ and $v_0$.  The result is then the following.

\begin{prop} If $\mathcal{M}_O(n,a,b)$ is feasible, $a>1$ and $[v_0,v_a]\in\mathcal{G}$, then it admits
an involutary automorphism swapping $\mathcal{G}$ and $\mathcal{R}$ if and only if
\begin{itemize}\item $8|n$,
\item $b=4b_0+1<(3n-4)/4$ with $b_0$ odd,
\item $4b_0^2\equiv 4(\mbox{\rm{mod} }n)$ and
\item $a=b-n/2+2$.
\end{itemize}
\end{prop} 

The conditions from the results in this section are compiled into a single entry for the
collection $\mathcal{C}$.


\begin{thebibliography}{9999}
\bibitem{A3} B. Alspach, Honeycomb Toroidal Graphs, {\sl Bull. Inst. Combin. Appl.}
{\bf 91} (2021), 94--114.
\bibitem{A2} B. Alspach and M. Dean, Honeycomb toroidal graphs are Cayley graphs, {\sl
Inform. Process. Lett.} {\bf 109} (2009), 705--708.
\bibitem{A1} B. Alspach, A. Khodadapour and D. L. Kreher, On factor-invariant graphs, {\sl Discrete
Math.} {\bf 342} (2019), 2173--2178.
\bibitem{B2}  W. Bosma, J. Cannon and C. Playoust, The MAGMA algebra system I: the user
language, {\sl J. Symbolic Comput.} {\bf 24} (1997), 235--265.
\bibitem{B1} I.Z. Bouwer (Ed.), {\sl The Foster Census}, Winnipeg, 1988.
\bibitem{C2} S. Cavior, The subgroups of the dihedral group, {\sl Math. Mag.} {\bf 48} (1975),
107--107. 
\bibitem{F2} Y. Q. Feng and K. Wang, $s$-regular cyclic coverings of the three-dimensional
hypercube $Q_3$, {\sl European J. Combin.} {\bf 24} (2003), 719--731.
\bibitem{F3} B. Frelih and K. Kutnar, Classification of cubic symmetric tetracirculants and
pentacirculants, {\sl European J. Combin.} {\bf 34} (2013), 169--194.
\bibitem{F1} R. Frucht, J. Graver and M. Watkins, The groups of the generalized Petersen graphs,
{\sl Proc. Cambridge Phil. Soc.} {\bf 70} (1971), 211--218.
\bibitem{J1} K. Jasen\v{c}\'akov\'a, R. Jajcay and T. Pisanski, A new generalisation of generalised
Petersen graphs, {\sl Art Discrete Appl. Math.} {\bf 3} (2020), \#P1.04.
\bibitem{S1} P. \v{S}parl, Symmetries of the Honeycomb toroidal graphs, {\sl J. Graph Theory},
to appear.
\bibitem{T1} W. Tutte, {\sl Connectivity in Graphs}, University of Toronto Press, Toronto, 1966.
  
\end{thebibliography}
\end{document}